\documentclass[12pt,oneside,reqno]{amsart}

\usepackage{geometry}
\usepackage{amsaddr}
\usepackage{enumerate}

\newtheorem{defin}{Definition}[section]

\newtheorem{thm}{Theorem}[section]
\newtheorem{atheo}{Theorem}

\newtheorem{remark}{Remark}[section]

\newtheorem{corollary}{Corollary}[section]

\newcommand{\be}{\begin{equation}}
\newcommand{\ee}{\end{equation}}
\newcommand{\baa}{\begin{array}}
\newcommand{\eaa}{\end{array}}
\newcommand{\ba}{\begin{eqnarray}}
\newcommand{\ea}{\end{eqnarray}}

\newcommand{\R}{\mathbb{R}}

\newcommand{\up}{\overline{\phi}}
\newcommand{\low}{\underline{\phi}}
\newcommand{\ds}{\mathbb{\displaystyle}}

\title[Asymptotic speed selection]{Speed Selection for Reaction Diffusion Equations in Heterogeneous Environments}
\author{Mohammad El Smaily $^{1}$, Chunhua Ou $^{2}$}
\address{$^{1}$ Department of Mathematics and Statistics\\University of Northern British Columbia,\\
Prince George, BC, Canada\\
$^{2}$ Department of Mathematics and Statistics\\Memorial University,\\ St. John's, NL, Canada}
\email{M. El Smaily: mohammad.elsmaily@unbc.ca, Chunhua Ou: ou@mun.ca}
\date{}
\begin{document}
\subjclass[2010]{35K55, 35Q92, 37N25}
\keywords{propagation speed, reaction-advection-diffusion, linear/nonlinear selection}

\begin{abstract}
{Reaction-advection-diffusion equations, in periodic settings and with general type nonlinearities, admit a threshold known as the minimal speed of propagation. The minimal speed does not have an accessible formula when the nonlinearity is not of KPP type, for instance. The  question becomes whether the minimal speed can be obtained through a linearization procedure or not. In this paper, we derive selection criteria for the minimal speed: a key feature of the nonlinear selection is unveiled. Moreover, we use  upper/lower solution techniques in order to derive practical criteria determining the minimal speed in the presence of advection and a general type nonlinearilty.}

\end{abstract}

\maketitle
%%%\tableofcontents

\section{Introduction and setting}

This paper is  concerned with the speeds of propagating wavefronts  for reaction-advection-diffusion equations in periodic media. The general form of such equations is 
\begin{equation}
u_t =\Delta u+q(x)\cdot \nabla u+f(x,u), \label{1}
\end{equation}
where $t\in\mathbb{R},~x \in\R^{N}$ and $N\geq1$ is the space
dimension. In order to describe the problem's setting briefly (mainly the advection term $q(x)$ and the reaction $f(x,u)$),  let
$L_1,\cdots,L_N$ be $N$ positive real numbers.  We   state the definitions of a periodicity cell and an $L$-periodic field as follows. 
 The set $$\mathcal{C}=\{x\in\R^{N}
\text{ such that } x_{1}\in(0,L_1),\ldots,x_{N}\in(0,L_N)\}$$ is called the
periodicity cell of $\R^{N}.$ A field
$w:\R^{N}\rightarrow\mathbb{R}^N$ is said to be $L$-periodic
 if $w(x_1+k_1,\cdots,x_N+k_N)=w(x_1,\cdots,x_N)$, almost everywhere in $ \R^{N}$  and for all
$\displaystyle{k=(k_1,\cdots,k_N)\in L_1\mathbb{Z}\times\cdots\times L_N\mathbb{Z}}.$

\smallskip
 In this work, the advection $q(x)=(q_1(x),\cdots,q_N(x))$  is a
vector field
satisfying
\begin{equation}\label{cq}
    \left\{
      \begin{array}{ll}
 
q \in C^{1,\alpha}(\R^{N}), \text{ for some }\alpha>0,&\vspace{7 pt}\\
        q\hbox{ is $L-$periodic with respect to }x, & \vspace{7 pt}\\
        \nabla\cdot q\equiv0\hbox{ in } \R^{N}.
 %%%%%%%       \forall 1\leq i\leq N, \quad\displaystyle{\int_{C}q_i{ \rm d}x =0}\hbox{.}
      \end{array}
    \right.
\end{equation}
The nonlinearity  $f=f(x,u)$, in \eqref{1}, is a
function defined in $\R^{N}\times[0,1],$ such that
\begin{equation}\label{f1}
        f\geq0, ~f\hbox{ is $L$-periodic with respect to } x,~ f\in C^{1,\alpha}(\R^{N}\times[0,1]),
        \end{equation}
        and
        \begin{equation}\label{f2}
\left\{\begin{array}{l}
     \forall x\in\R^{N},\quad \displaystyle{f(x,0)=f(x,1)=0 } \hbox{,}\vspace{10 pt} \\
        \exists \rho\in(0,1)\text{ such that }\forall \,x\in \R^{N},\displaystyle{\forall\, 1-\rho \leq s\leq
s'\leq1,}
~~\displaystyle{f(x,s)\geq f(x,s') } \hbox{,}\vspace{10 pt} \\
        \forall s\in(0,1), \exists\, x\in \R^{N}\hbox{ such that }f(x,s)>0  \hbox{.}
      \end{array}
    \right.
\end{equation}
An example of such nonlinearity is \[f(x,u)=b(x)\left[u(1-u)(1+a(x)u)\right]\quad u\in [0,1]~,x\in \R^N,\tag{E}\] where $a,b:\R^N\rightarrow\R$ can be taken as a smooth, periodic functions, with $a(x)\geq0$ and $b(x)\geq0$ for all $x\in \R^N.$ A more particular family of such nonlinearities is the well KPP/FKPP type (after Fisher–Kolmogorov–Petrovsky–Piskunov), which we describe in \eqref{linear}, below.

Under the above assumptions, we are interested in the minimal speed (or the spreading speed) of a specific kind of solutions, known as pulsating traveling fronts. In Definition \ref{define.pulsating}, we recall the definitions of both the minimal speed and a pulsating traveling wave/front, as introduced in Berestycki, Hamel \cite{BH1} and Xin \cite{Xin4}. 
\begin{defin}[\cite{BH1}, \cite{Xin4}]\label{define.pulsating}
Let $e=(e^1,\cdots,e^N)$ be an arbitrarily  unit direction in
$\mathbb{R}^N.$ A function $u=u(t,x)$ is called a pulsating
traveling front propagating in the direction of $e$, with an
effective speed $c\neq0,$ if $u$ is a classical solution of
\begin{eqnarray}\label{front}
      \begin{array}{ll}
u_t =\Delta u+q(x)\cdot \nabla u+f(x,u),~
t\in\mathbb{R},~x \in\R^{N},
\vspace{6 pt}\\
 \displaystyle{\forall k\in L_1\mathbb{Z}\times\cdots\times L_N\mathbb{Z},~ \forall(t,x)\in\mathbb{R}\times{\R^{N}}},
~\displaystyle{u(t+\frac{k\cdot e}{c},x)=u(t,x+k)}  \hbox{,} \vspace{6 pt}\\
         \displaystyle{\forall t\in\mathbb{R}, \lim_{x\cdot e\rightarrow-\infty}u(t,x)=0\hbox{ and } \lim_{x\cdot e\rightarrow +\infty} u(t,x)=1}  \hbox{,}
        \vspace{6 pt} \\
          0\leq u\leq1,
      \end{array}
\end{eqnarray}
 where the above limits hold locally in $t$ and uniformly in the
directions of $\mathbb{R}^N$ that are orthogonal to $e$.
\end{defin}
\noindent Note that Definition \ref{define.pulsating} can be rephrased upon using a traveling wave variable $s:=x\cdot e+ct$
and plugging the ansatz $$u(t,x):=\phi(x\cdot e+ct,x)=\phi(s,x)$$ in \eqref{front}. From this, we learn that a pulsating traveling wave $\phi$  is $L$-periodic in $x$ (namely, from the second line in \eqref{front}) and  satisfies the equation
\begin{equation}\label{ceq}
\Delta_{x}\phi+\phi_{ss}+2e\cdot\nabla_{x}\phi_{s}
+q\cdot\nabla_{x}\phi+(q\cdot e-c)\phi_s+f(x,\phi(s,x))=0,\end{equation}
 for all $(s,x)\in \R\times \R^{N}.$  Furthermore,  defining $L_{c}$ to be the operator
\begin{equation}\label{degen_elliptic}
L_c\phi:=\Delta_{x}\phi+\phi_{ss}+2e\cdot\nabla_{x}\phi_{s}
+q\cdot\nabla_{x}\phi+(q\cdot e-c)\phi_s~~\hbox{in}~\mathbb{R}\times \R^{N},
\end{equation}
we obtain that a pulsating
traveling front $\phi(s,x)$ satisfies the wave profile equation
\begin{equation} L_{c}\phi+f(x,\phi)=0, \label{wave}
\end{equation}
subject to the limiting boundary conditions
\begin{equation}
\lim_{s\rightarrow-\infty}\phi(s,x)=0 \hbox{ and } \lim_{s\rightarrow +\infty} \phi(s,x)=1\text{ uniformly in $x \in\R^{N}$.} \label{bc}
\end{equation}

Existence of pulsating traveling fronts, for this class of equations, is well studied and the above discussion is only a brief introduction, which is by no means exhaustive. We recall the most relevant existence results in Theorem \ref{Berest_Hamel_ZFK}, below. We refer the reader to \cite{BH1}, and the references therein, for complete details.

\begin{atheo}[Berestycki, Hamel \cite{BH1}]\label{Berest_Hamel_ZFK}
Let $e$ be any unit vector in $\mathbb{R}^{N}.$ Assume that $q$ satisfies (\ref{cq})
 and let $f$ be a nonlinearity
satisfying \eqref{f1} and \eqref{f2}. Then,
there exists $\displaystyle{c^{*}>0}$ such that
the problem (\ref{wave})-(\ref{bc}) has no solution $(c,\phi)$ if
$c<\displaystyle{c^{*}}$ while, for each
$c\geq\displaystyle{c^{*}},$ it has a pulsating traveling front solution
$(c,\phi)$ such that $\phi$ is increasing in $s.$
\end{atheo}

\section*{Linearization, Important facts and Statement of the problem}

\noindent Theorem \ref{Berest_Hamel_ZFK} applies in a general periodic framework and   provides the existence of fronts and a threshold $c^*$. However,  only variational type formulas for $c^{*}$ are available, when the nonlinearity $f$ in \eqref{front} satisfies  the general conditions \eqref{f1} and \eqref{f2}.  For instance, a min-max formula for $c^*$ (that holds under conditions \eqref{f1} and \eqref{f2} on $f$) is proved in Theorem 1.9 of  \cite{MINMAX}.

Let us recall now recall an attempt to estimate the speed $c^*$, given in Theorem \ref{Berest_Hamel_ZFK} above, when $f$ is differentiable with respect to $u$ at $u=0$. For convenience, we write $\eta(x)=\frac{\partial f}{\partial u}(x,0)$. The linearized version of (\ref{wave}), at $\phi=0$ (equivalently $u=0$),   reads
\begin{equation} L_{c}\phi+\eta(x)\phi=0. \label{linearwave}
\end{equation}
 Hamel \cite{H} introduced what we will call the \emph{linear} speed, and  denote by $c_0$, through the variational formula 

\begin{equation}\label{linearspeed}
    \displaystyle{c_0(e)=c_{0}^{q,f}(e)=\min_{\lambda>0}\frac{k(\lambda)}{\lambda}},
\end{equation}
where $\displaystyle{k(\lambda)=k_{e,q,\eta,\lambda}}$ is
the principal eigenvalue of the elliptic operator
$\displaystyle{L_{e,q,\eta,\lambda}}$  defined by
\begin{equation}\label{Leq}
\displaystyle{L_{e,q,\eta,\lambda}\Psi:=}\displaystyle{\Delta\Psi+2\lambda e\cdot
\nabla\Psi+q\cdot\nabla\Psi}\displaystyle{+[\lambda^2+\lambda
q\cdot e+\eta]\Psi},
\end{equation}
acting on the space
\begin{eqnarray*}
\begin{array}{ll}
E=&\left\{ \Psi \in C^2(\overline{\Omega}), \Psi\hbox{ is
$L$-periodic with respect to $x$}
\right\}.
\end{array}
\end{eqnarray*}
Note that the elliptic operator in \eqref{Leq} is not a self-adjoint due to the presence of the drift term $q$ in our problem. The principal eigenfunction is positive and unique up to multiplication by a constant. 

 A detailed study of the properties of $k(\lambda)$ is done in \cite{BH1} and \cite{BHN1}. In particular, \cite{BH1} shows that $\lambda\mapsto k(\lambda)$ is a convex function.  Note that a lower bound for $k(0)$ can be given by  \begin{equation}\label{lower}
 k(0)\geq \min_{\R^N}\eta(x)=\min_{\mathcal{C}}\eta(x).
\end{equation} The proof of lower bound \eqref{lower} is given  in the footnote, for the reader's convenience.\footnote{ For $\lambda=0$, we have the principal eigenfunction (denote by $\varphi$ and normalized by $\|\varphi\|_{L^2(\mathcal C)}=1$) satisfies  $\Delta\varphi+q\cdot \nabla\varphi+\eta(x)\varphi=k(0)\varphi,$ $\varphi>0$ and $L-$periodic. Multiplying by $\varphi$ and integrating by parts over the periodicity cell $\mathcal C$, and because $\nabla\cdot q=0$,  we get $-\int_C|\nabla \varphi|^2+\int_\mathcal{C}\eta(x)\varphi^2=k(0)$. This gives us a lower bound for $k(0)$.} This lower bound guarantees that $k(0)\geq0$, as our nonlinearity $f$ is nonnegative. If we further assume  that \begin{equation} k(0)>0 \label{k0},\footnote{ The lower bound \eqref{lower} shows that  the assumption \eqref{k0} holds automatically whenever $\min_{\mathcal C}\eta>0$}
\end{equation}
 we then obtain a unique  $\lambda =\bar \mu$ such that
\begin{equation}\label{var}
    \displaystyle{c_0(e)=\frac{k(\bar \mu)}{\bar \mu}}.
\end{equation}
Moreover, for $c>c_{0}$, the equation
\begin{equation}\label{mu2a}
  c\lambda=k(\lambda)
\end{equation}
then admits two solutions $\lambda=\mu_1(c)$ and $\lambda=\mu_2(c)$, with $\mu_1(c)<\mu_2(c)$. When $c=c_0$, we have $\mu_1(c)=\mu_2(c)=\bar \mu$. By appealing to the convexity
of the function $k(\lambda)$ again, we get that $\mu_1(c)$ is \emph{decreasing} in $c$ and $\mu_2(c)$ is \emph{increasing} in $c$.

\subsection*{The relation between $c_0$ and $c^*$}   We return now to the influence of the nonlinearity $f$ on the speeds $c_0$ and $c^*$.  To do this, we stop by  the particular type of KPP nonlinearities. We  say that $f$ is of KPP type if $f$ satisfies \eqref{f1}, \eqref{f2} and the 
 \emph{additional} KPP condition 
 \begin{equation}\label{linear}
   0<f(x,u)\le \eta(x) u, \text{ for all }u\in(0,1). 
 \end{equation}
Note that \eqref{linear} already assumes $\eta(x)>0$. Thus, when $f$ is of KPP type, we directly get $k(0)>0$ (see the lower bound \eqref{lower} of $k(0)$).  
 A major difference between the KPP class and a nonlinearity satisfying only \eqref{f1} and \eqref{f2} is the sublinearity at $u=0$ (i.e. \eqref{linear}). For example, $f(u)=u(1-u)(1+au)$ satisfies \eqref{f1} and \eqref{f2} but does not satisfy \eqref{linear}, when $a>2$, for instance. Also, a KPP nonlinearity must be positive everywhere in $\R^N\times(0,1).$ This need not be the case for the class (E), which we mention above (for e.g, take $a(x)\equiv 3$, $b(x)=\sin^2|x|$ and set $f(x,u):=b(x)u(1-u)(1+3u)$).

  In the particular case, where the nonlinearity $f$ satisfies  \eqref{f1} and \eqref{f2}, together with the KPP condition \eqref{linear}, Berestycki, Hamel and Nadirashvilli \cite{BHN1} proved  that the minimal speed $c^*$, in Theorem A, is exactly equal to $c_0$ in (\ref{linearspeed}) (also see \cite{H}). However, for a more general nonlinearity  $f$, which  satisfies conditions \eqref{f1} and \eqref{f2} only, it is still unknown how the minimal speed is determined (aside from variational formulas in \cite{MINMAX}, for example). From Theorem \ref{Berest_Hamel_ZFK} and the fact that a KPP type nonlinearity satisfies \eqref{f1} and \eqref{f2}, we can see  that $c^*\geq c_0$ holds always. The primary purpose of this paper is to  investigate the comparison of $c_0$ to $c^*$ further.  We prove that, when the minimal speed $c^*$ is greater than the linear speed $c_0$, the corresponding wave front (\emph{pushed front}) decays with a faster rate; this solves the conjecture in \cite[page 363]{H}. 
To speak about other goals of this work, we recall the following definition, which has been used in the literature (see \cite{Lucia_2004}, for instance).
\begin{defin}[Linear and nonlinear selection mechanisms]\label{selection} Under the assumptions of Theorem \ref{Berest_Hamel_ZFK}, we call the case $c^*=c_0$ the linear selection mechanism and the case $c^*>c_0$ the nonlinear selection mechanism.
\end{defin}
\noindent In this context, we will use the upper/lower solution method in order to provide an easy-to-use approach that determines whether the minimal speed is selected linearly or nonlinearly (see Definition \ref{selection} above). In the case of nonlinear selection, we will show a method that  leads to a lower or an upper bound estimate of the minimal speed. We show  our main results in sections \ref{resultats} and \ref{apps}. Section \ref{apps} serves as an application of the  theorems in Section \ref{resultats}.

\section{Pushed wavefront}\label{resultats}
For a given wavefront $\phi$, satisfying (\ref{wave}) with  $c>c_0$, a straightforward derivation of the characteristics of the linear part of the wave profile proposes that
either
\begin{equation}
\phi(s,x)\sim C_1 \Psi_{\mu_1} (x) e^{\mu_1(c)s}, \,\, C_1>0, \label{decay1a}
\end{equation}
or \begin{equation}
\phi(s,x)\sim C_2 \Psi_{\mu_2}(x)e^{\mu_2(c)s}, \,\, C_2>0 \label{decay2b}
\end{equation} as $s \to -\infty$, where $\Psi_{\mu_i} (x), i=1,2,$ is the eigenfunction corresponding to  the principal eigenvalue $k(\mu_i)$ defined in (\ref{Leq}). For a rigorous proof of this property, we refer the reader to \cite{H}.

Alternatively, when linearizing  the first equation of (\ref{front}) at $u=0$, we  obtain the linear partial differential equation
\begin{equation}
u_t =\Delta u+q(x)\cdot \nabla u+\eta(x)u,  \text{ where } \eta(x)=\partial_uf(x,0).  \label{linearflow}
\end{equation}
The above equation  defines a linear semiflow $M(u_0)=u(t,x,u_0)$, where $u_0$ is the initial data. Obviously, we have
\begin{equation}
M(\Psi_{\mu_i} (x) e^{\mu_i(c)x\cdot e})=\Psi_{\mu_i} (x) e^{\mu_i(c)[x\cdot e+ct]},  \quad  i=1,2.
\end{equation}
\subsection{Fast decay nature of the pushed wavefront}

\begin{thm}[Necessary and sufficient condition]\label{them1}
Assume that (\ref{k0}) holds  and let $\phi_{c^{*}} (s , x)$ be the wavefront of (\ref{wave}), with the speed  $c^*$ (the minimal speed). Consider the linear speed   $c_0$   defined in (\ref{var}).   The following results hold:
\begin{enumerate}[(i)]
\item If there
exists a speed $c=\bar{c}$ $>c_0$, such that  (\ref{wave}) has a
non-decreasing traveling wave solution $\phi_{\bar{c}}(s,x)$, connecting $0$ to $1$ and satisfying the asymptotic  behavior
\begin{equation}
\phi_{\bar{c}}(s,x)\sim C\Psi_{\mu_2(\bar{c})}(x)e^{\mu _{2}(\bar{c})s}\text{ as }
s\rightarrow -\infty   \label{3.5},
\end{equation}
 where $\mu_2$ is defined in (\ref{mu2a}) and  $C$ is an arbitrary positive constant, then we have $c^*=\bar{c}>c_0.$ In other words, the minimal speed $c^*$ is nonlinearly selected.
\item If the spreading speed $c^{*}$ is nonlinearly selected {\rm(}i.e. $c^*>c_0${\rm)}, then the wave front $\phi_{c^{*}} (s , x)$ has the fast decay behavior defined in (\ref{decay2b}):  \begin{equation}
\phi_{c^{*}}(s,x)\sim C_2 \Psi_{\mu_2}(x)e^{\mu_2(c^*)s} \text{ as } s \to -\infty, \text{ for some } C_2>0. \label{decay2bb}
\end{equation}
\end{enumerate}
\end{thm}
\begin{proof} (i). We first prove part one. Suppose that there is a traveling wave with speed $c=c'<\bar c$. Then, by Theorem 1.5(a) in \cite{H}, we have a contradiction with (\ref{3.5}). This contradiction implies that the minimal wave speed is nonlinearly selected. A more direct proof, under certain assumption, is provided in Remark \ref{alt.proof}, below.

(ii) For the second part,
instead of the definition of wavefront in (\ref{wave})-(\ref{bc}),  we can alternatively rephrase the definition of  a pulsating traveling wave in terms of semiflow, as done in Liang and  Zhao \cite{liang1}.  Assume that $Q(u_0)=u(t,x,u_0)$ is the solution semiflow induced by (\ref{1}), with the initial function $u_0(x)$, to be continuous, nonnegative and bounded. A traveling wave solution $\phi(s,x)$, with $\phi(-\infty,x)=0,\, \phi(\infty,x)=1$, should  then satisfy
\begin{equation}
Q[\phi(x\cdot e,x)]=\phi(x\cdot e+ct,x). \label{semi-wave}
\end{equation}
Due to the Laplacian operator in the equation, one can easily get that the semiflow $Q$ is compact and strongly positive.

 We assume that the minimal speed $c^*$ is nonlinearly selected; that is, $c^*>c_0$. We proceed  to show that at the speed $c=c^*$, the traveling wave $W_{c^*}(s,x)$ satisfies
\begin{equation}
W_{c^*}(s,x)\sim C \Psi_{\mu_2(c^*)}(x)e^{\mu_2(c^*)s} \text{ }   \    \text{as }  \   \   s \to -\infty,
\end{equation}
for some constant $C$.
By the alternatives (\ref{decay1a}) and (\ref{decay2b}), assume to the contrary that
\begin{equation}
W_{c^*}(s,x)\sim C_3 \Psi_{\mu_1(c^*)}(x) e^{\mu_1(c^*)s} \   \text{ }  \text{as }  \   \  s \to -\infty, \label{first}
\end{equation}
for some positive constant $C_3$ and eigenvector  $\Psi_{\mu_1(c^*)}$. We will  prove that  the operator $Q$ has a traveling wave $W_c(x\cdot e,x)$ satisfying
\begin{equation}
Q(W_c)=W_c(x\cdot e+ct,x) ~\text{ or } ~ T_{ct}Q(W_c)=W_c ,\label{*}
 \end{equation}
 for some speed $c=c^*-\delta$, where $T_{ct}W(s)=W(x-ct)$ is the right-shifting operator   and  $\delta$ is a sufficiently small and positive number. Hence, $c^*$ is not the minimal speed and this will lead to a contradiction.

 Indeed, under assumption (\ref{first}), we define
 \begin{equation}
 \bar W=W_{c^*}(s,x)\omega(s,x),\text{ where  }  \omega(s,x)=\frac{1}{1+ \frac{\Psi_{\mu_1(c)}(x)}{\Psi_{\mu_1(c^*)}(x)}\delta e^{-(\mu_1(c)-\mu_1(c^*))s}}.
 \end{equation}
  Note that when $\delta$ small, $\bar W$ is close to  $W_{c^*}$, but with a different decaying rate at $-\infty$. We will  show the existence of a solution to (\ref{*}), provided that $\delta$ is sufficiently small. In (\ref{*}), seek a $W_c$ of the form 
\begin{equation}
W_c=\bar W+V, \label{2.15}
\end{equation}
such that 
\begin{equation}
T_{ct}Q(\bar W+V)=\bar W+V,
\end{equation}
where $V=V(s,x)$ is a function to be determined. A straightforward  calculation leads us to
\begin{equation}
V=T_{c^*t}M(W_{c^*})V+F_{0}+M_{\delta}V+F_{high}(V) \label{V},
\end{equation}
where
\begin{equation}
F_{0}=T_{ct}Q(\bar W)-\bar W,
\end{equation}
\begin{equation}
M_{\delta}V=\left[T_{ct}M(\bar W)-T_{c^*t}M(W_{c^*})\right]V
\end{equation}
and
\begin{equation}
F_{high}(V)=T_{ct}Q(\bar W+V)-T_{ct}Q(\bar W)-T_{ct}M(\bar W)V.
\end{equation}
Here, $M(\bar W)$ is the Fr$\acute{e}$chet derivative of $Q$ around the function $\bar W$. With  a simple estimate, it follows  that $M_{\delta}V=O(\delta)V$ and $F_0=O(\delta)$, where  
$$
F_0=o(e^{\mu_1(c^*)s}) \    \     \text{as}  \   \   s \to -\infty.
$$
For a solution to (\ref{V}), we recall that $M(W_{c^*})$ is defined by
$$
M(W_{c^*})[V]=\lim_{\rho\rightarrow 0}\frac{Q[W_{c^*}+\rho V]-Q[W_{c^*}]}{\rho},
$$
for $V$ in the space $C_0:=\{ u \in  C(\mathbb{R} \times [0,L],\mathbb{R}): u( \pm \infty,x)=0\}.$   The operator $T_{c^*t}M(W_{c^*})$ is compact and  strongly positive, its principal eigenvalue is $\lambda=1$ and the corresponding principal eigenvector is $\bar {v}=W'_{c^*}$. It is not difficult to see  that $W'_{c^*}$ shares the same decaying behavior as $W_{c^*}$. That is,

\begin{equation}
W'_{c^*} \sim {C(x)}  e^{\mu_1(c^*)s}     \    \    \text{ }   \text{as } \     \     s \to -\infty,  \label{second}
\end{equation}
for some periodic function ${C(x)}$, where $W'_{c^*}$ represents the first derivative of $W_{c^*}(s,x)$ with respect to $s$.

 Next, in order to omit the eigenvector   $\bar{v}$, we construct  a weighted space $\mathcal{V}$ as
$$
\mathcal {V}=\{ v \in C_0: v e^{-\mu_1(c)s}=o(1) \text{ as } s \to -\infty                       \},
$$
where $c=c^*-\delta$. Consequently, we see  that the eigenvector $\bar{v}=W'_{c^*}$ is not in $\mathcal{V}$, and this rules out $\lambda=1$ of being an eigenvalue for $T_{c^*t}M(W_{c^*})$ defined on $\mathcal{V}$. Since the operator
    $T_{c^*t}M(W_{c^*})$ is compact and  strongly positive in $\mathcal{V}$, it follows  that $T_{c^*t}M(W_{c^*})-I$ has a bounded inverse in $\mathcal{V}$, where $I$ is the identity operator. Using  the  inverse function theorem in the space $\mathcal{V}$, we obtain  a small positive number $\delta_0$ so that problem (\ref{V}) has a solution $V$ for any $\delta \in [0,\delta_0)$. 
    
    Now, we have a solution $W_c$, for $c=c^*-\delta$, as desired in  (\ref{2.15}). The positivity of $W_c$ is guaranteed by the choice of a sufficiently small $\delta$ (smaller than $\delta_0$) and this completes the proof.
\end{proof}
\begin{remark}[A more direct proof of Theorem \ref{them1}, Part (i), provided exponential stability of the positive equilibrium]\label{alt.proof}
In the proof of the  first part of Theorem \ref{them1}, we have made use of  Theorem 1.5 (a) in Hamel \cite{H} of \cite{H}. Actually, we can give a more direct  proof of the latter,  in the case where  the positive equilibrium ( $u=1$, in our setting) is exponentially stable. The proof techniques are based on linearization at $u=1$ and semiflows. 

\begin{proof}[Proof of the statement in Remark \ref{alt.proof}] Suppose that (\ref{3.5}) is true. We  proceed   to prove that (\ref{wave}) has no traveling waves for any $c$ in  $(c_0, \bar{c})$. To the contrary, suppose  that for some $c \in (c_0, \bar{c})$, there exists  a traveling wave $W_c(x\cdot e,x)$ satisfying either (\ref{decay1a}) or (\ref{decay2b}).
In view of the monotonicity of $\mu_1(c)$ and $\mu_2(c)$ in $c$, we get  $W_c(s,x) > \phi_{\bar{c}}(s,x),$ for $s$ near $-\infty.$

To understand the behavior of this solution near $+\infty$, 
     let     $\bar{k}(-\gamma)$ be the principal eigenvalue of the linear operator $\displaystyle{L_{e,q,\eta,-\gamma}}$ defined in (\ref{Leq}) (but with $\eta$ replaced by $\zeta(x):=\partial_uf(x,1)$). By linearizing equation (\ref{wave}) at $1, $ we obtain a characteristic equation  $-\gamma c-\bar{k}(-\gamma)=0$. We assume that
       \begin{equation}
       \bar{k}(0)<0. \label{stable}
       \end{equation}
   We emphasize that assumption \eqref{stable} guarantees the convexity of $\gamma\mapsto\bar k(-\gamma)$ and is sufficient for the exponential  stability of the positive equilibrium 1 (see Hamel \cite{H}, page 364). 
 Based on the convexity of $\bar{k}(-\gamma)$, we can find a unique positive $\gamma$ that solves the characteristic equation. Moreover, $\gamma$ is a decreasing function in $c,$ whenever $c\ge c_0$. This yields to
\begin{equation}
W_c\sim 1-\Psi_{\gamma}(x)e^{-\gamma x},
\end{equation}
for some positive $\gamma$ and  positive function $\Psi_{\gamma}(x)$. In view of the monotonicity of $\gamma$ in $c$, we  further obtain that $\phi_{\bar c}(s,x) \ll W_c(s,x)$ for $s$ near $\infty$. Therefore, it is always possible to  make a shift of  distance $\xi_0$ for the variable $s$ in $W_c(s, x)$ such that
$$
\bar{W}_c(x\cdot e,x)=W_c(x\cdot e+\xi_0,x) > \phi_{\bar c}(x\cdot e,x).
$$
The monotonicity of the map $Q$ implies that 
\begin{equation}
\bar{W}_c(x\cdot e+ct,x)=Q(\bar{W}_c(x\cdot e,x))\ge Q(\phi_{\bar c}(x\cdot e,x))=\phi_{\bar c}(x\cdot e+\bar c t,x) \label{cc}
\end{equation}
 On the other hand, on the line $x\cdot e+t \bar c=z_0$,  for some fixed value $z_0$, it follows that  $\phi_{\bar c}(x\cdot e+\bar c t,x)=\phi_{\bar c}(z_0,x)> 0$ and
$$
\bar{W}_c(x\cdot e+ct,x)=\bar{W}_c(z_0-t(\bar c-c),x) \to 0 \text{ as } t \to +\infty.
$$
The latter  contradicts (\ref{cc}). As such, there exist no  traveling waves for $Q$ when $c\in (c_0, \bar c)$. It follows now, from Theorem A, that we cannot have traveling waves with speed $c=c_0$. This provides an alternative proof to the first part of Theorem \ref{them1}, under the assumption that exponential stability of the positive equilibrium holds. \end{proof}
\end{remark}
\begin{remark}[More on the proof of Theorem \ref{them1}]
In the degenerate case, where $\bar{k}(0)=0$, we speculate that the above idea and argument still work, as long as we can show that the traveling wave solution is non-increasing in $c$, for large $s$ ($ s \to+ \infty$). This could be done by constructing an upper solution $\hat \phi=1$  and a lower solution $ \phi= \phi_{\bar c}$ for (\ref{wave}), with $\bar c>c \text{ and } s\ge s_0$, where $s_0$ is a given constant. The uniqueness of the wavefront (up to translation)  may be of use then.  We will leave this idea to interested readers.
\end{remark}
\begin{remark}
The second part of our Theorem \ref{them1} confirms the conjecture in \cite[page, 363]{H}.
\end{remark}

\begin{remark}[More accessible criteria]
Although we have unveiled the important feature of pushed wavefronts in Theorem \ref{them1},  we cannot practically establish linear/nonlinear selection criteria by  Theorem \ref{them1}. This is because exact traveling wave formulas are unknown. To this end, Sections \ref{independent.linear} and \ref{independent.nonlinear}, below, will be dedicated to develop certain easy-to-apply formulas that determine the speed selection mechanism. The formulas are based on constructions of upper or lower solutions that  approximate the exact traveling waves to some extent. The establishment of these criteria does not rely on Theorem \ref{them1} and can be of independent interest to readers.
\end{remark}

\subsection{Linear selection}\label{independent.linear}

\begin{thm}[Linear Selection]\label{them2}
Let    $c_0$ be as defined in (\ref{linearspeed}).  Further, assume that there exists a continuous and positive function $U(s,x)$ satisfying

\begin{equation} L_{c_0}U+f(x,U)\le 0, \label{c1}
\end{equation}
together with
\begin{equation}
\liminf_{x\rightarrow \infty} U(s,x)> 0 \text{ and } \lim_{s\rightarrow -\infty} U(s,x)= 0.\label{leftbound}
\end{equation}
  Then, the linear selection is realized. That is, $c^*(e)=c_0(e).$
\end{thm}
\begin{proof} Similar to what is done in \cite{fy} and \cite{liang1}, we can define the leftward spreading speed $c^*$ as
\begin{equation}\label{spreadingspeed}
  c^*:=\sup\{c:\lim_{i\to -\infty,  i\in \mathbb{Z}}a(c;iL+\theta)=1, \theta \in[0,L] \}
\end{equation}
where $$a(c;x)=\lim_{n\rightarrow\infty} a_n(c;x).$$
In our setting, for
a given  real number $c$, the  sequence of functions $\{a_n\}_{n=0}^{\infty}$ is defined as
\begin{equation}
  a_0(c;x)=\phi(x), \quad  a_{n+1}(c;x)=R_c[a_n(c;\cdot)](x), \label{2b}
  \end{equation}
 and
    \begin{equation}
  R_c[a_n](x)=\max \{\phi(x),T_{c}[Q_1[a_n]](x)\},   \label{2a}
  \end{equation}
where $\phi(x)$ is non-decreasing function that satisfies   $$\phi(x)=0\text{ for }x\le 0~\text{ and }\lim_{x\to +\infty} (\phi(x)-\omega)=0,$$  $0<\omega<1$, and $Q_1$ is the solution semiflow $Q_t$ at $t=1$.
Here, the limit in (\ref{spreadingspeed}) is obtained by splitting the variable $x$  interval-by-interval with each interval length as $L$.  $c^*$ is independent of the choice of $ \phi$, see \cite{fy,liang1}.  Therefore, we can let $ \phi(\infty)$ be small so that the upper solution $U$ (or a shift of $U$ if needed) satisfies
\begin{equation}
a_0(c_0;x) \le U(x\cdot e,x)
\end{equation}
for all $x\in (-\infty,\infty)$. From (\ref{2a}), (\ref{2b}), (\ref{leftbound}) and (\ref{c1}), by induction, it follows that
\[
a_{n+1}(c_0;x)\le U(x\cdot e,x),  \quad n\ge 0.
\]
Thus, $a(c_0;-\infty)=0$.  By (\ref{spreadingspeed}), we have $c^*\le c_0$.   Therefore,  we arrive at $c^*=c_0$ by Theorem A, and  the linear selection is realized.
\end{proof}

\begin{corollary} \label{them3}
   Suppose that $f(x,u)\le f'(x,0)u$.
 Then, the linear selection is realized.
 \end{corollary}
 \begin{proof}For $c=c_0$, one can easily verify that  $U=e^{\bar{\mu}s}\Psi_{\bar{\mu}}(x)$ is an   upper solution of the wave profile equation, where $\bar{\mu}$ is defined in (\ref{var}).
 \end{proof}

\begin{corollary} \label{co2.2} Let \begin{equation}\label{uppersol}
\overline\phi (s,x):=\ds{\frac{\Psi(x)}{\Psi(x)+e^{-\mu_{1}s}}},
\end{equation}
where $\Psi$ is the principal eigenfunction of \eqref{Leq} corresponding to $\lambda=\mu_{1}=\bar \mu$ and the principal eigenvalue $k(\mu_{1})=\mu_{1}c$,  when $c=c_0$.  Then, the minimal speed  is linearly selected if
\begin{equation}\label{cor2}
  -2\mu_{1}^{2}\up \Psi -2\up\frac{|\nabla \Psi|^{2}}{\Psi} -4\mu_{1}\up ~e\cdot\nabla\Psi +\frac{\Psi f(x,\up)}{\up\left(1-\up\right)}-\eta(x)\Psi\le 0
\end{equation}
\end{corollary}

\begin{proof}
 We compute
\begin{equation}\label{prop}\up_{s}(s,x)=
\mu_{1}\up\left(1-\up\right), \quad \up_{ss}(s,x)=\mu_{1}^{2}\up\left(1-\up\right)(1-2\up)\end{equation}
and \[1-\up(s,x)=\frac{e^{-\mu_{1}s}}{\Psi(x)+e^{-\mu_{1}s}}\text{ for all }(s,x)\in\R\times\R^{N}.\]
Also, \[\nabla_{x}\up(s,x)=\ds{\frac{e^{-\mu_{1}s}\nabla \Psi(x)}{\left(\Psi(x)+e^{-\mu_{1}s}\right)^{2}}}=\up\left(1-\up\right)\frac{\nabla \Psi(x)}{\Psi(x)},\]
which leads to
\[\nabla_{x}\up_{z}=\mu_{1}\up\left(1-\up\right)(1-2\up)\frac{\nabla\Psi}{\Psi}.\]
Moreover, \[\begin{array}{rl}\Delta_{x}\up=&\ds{ (\nabla\up-2\up\nabla_{x}\up)\cdot \ds{\frac{\nabla\Psi}{\Psi}}+\up\left(1-\up\right)\frac{\Delta\Psi}{\Psi}\frac{|\nabla\Psi|^{2}}{|\Psi|^{2}}}\vspace{10 pt}\\
=&\up(1-\up)(1-2\up)\ds{\frac{|\nabla\Psi|^{2}}{|\Psi|^{2}}}+\up(1-\up)\frac{\Delta\Psi}{\Psi}+\up(1-\up)\ds{\frac{|\nabla\Psi|^{2}}{|\Psi|^{2}}}.
\end{array}\]
Now, we substitute the above quantities in $L_{c}\up+f(x,\up)$ to obtain
\begin{equation}
\begin{array}{l}
L_{c}\up+f(x,\up)
\vspace{10pt}\\
=\ds{\frac{\up\left(1-\up\right)}{\Psi}}\left\{\mu_{1}^{2}(1-2\up)\Psi-2\up\frac{|\nabla\Psi|^{2}}{\Psi}+\Delta\Psi \right.\\
\left. +2\mu_{1}(1-2\up)e\cdot\nabla\Psi+q\cdot\nabla\Psi+\mu_{1}q\cdot e\Psi-c\mu_{1}\Psi+\frac{\Psi f(x,\up)}{\up\left(1-\up\right)}\right\}\vspace{10pt}\\
=\ds{\frac{\up\left(1-\up\right)}{\Psi}}\left\{k(\mu_{1})\Psi-c\mu_{1}\Psi-2\mu_{1}^{2}\up \Psi -2\up\frac{|\nabla \Psi|^{2}}{\Psi}\right.\\
\left. -4\mu_{1}\up ~e\cdot\nabla\Psi +\frac{\Psi f(x,\up)}{\up\left(1-\up\right)}-\eta(x)\Psi\right\}\vspace{10 pt}\\
=\ds{\frac{\up\left(1-\up\right)}{\Psi}}\left\{-2\mu_{1}^{2}\up \Psi -2\up\frac{|\nabla \Psi|^{2}}{\Psi} -4\mu_{1}\up ~e\cdot\nabla\Psi+ \frac{\Psi f(x,\up)}{\up\left(1-\up\right)}-\eta(x)\Psi\right\}.
\end{array}
\end{equation}
In the last line of the above equation, we used $k(\mu_{1})-c\mu_{1}=0.$ Therefore, $\bar \phi$ is an upper solution, when $c=c_0$.  Appealing to Theorem \ref{them2}, the proof is complete.
\end{proof}

\subsection{Nonlinear selection}\label{independent.nonlinear}

 \begin{thm}[Nonlinear selection]\label{them5}
   For $c_1>c_0$, suppose that there exists a function $V(s,x)$ satisfying
   \begin{equation}
  0 < V(s,x) < 1, ~~   \limsup_{s\rightarrow \infty} V(s,x) < 1, ~~ V(s,x)=\Psi_{{\mu_2(c_1)}}(x)e^{\mu_2(c_1)s}   ~ \text{ as }~ s\rightarrow -\infty     \label{non1}
   \end{equation}
   and
   \begin{equation} L_{c_1}V+f(x,V)\ge 0, \label{c1a}
\end{equation}
  where $\mu_2(c_1)$ is defined in (\ref{mu2a}). Then, $c^* \geq c_1$ and no traveling waves exist for $c\in[c_0,c_1)$. In other words, the nonlinear selection is realized.
  \label{them8}
   \end{thm}

   \begin{proof} The proof of this theorem is similar to that  of Remark 2.1. However, it is important to note that, with the second condition in (\ref{non1}), the condition $\bar k(0)<0$ is no longer needed.  With this note, we omit the details of the proof of Theorem \ref{them5}. 
   \end{proof}
\begin{corollary}\label{them6}
For $c=c_0+\varepsilon$, where $\varepsilon$ is a sufficiently small number,  let \begin{equation}\label{lowersol}
\low (s,x):=\ds{\frac{\Psi(x)}{\Psi(x)+e^{-\mu_{2}s}}}.
\end{equation}

\begin{equation}\label{non}
 \text{ If }~~ -2\mu_{2}^{2}\low \Psi -2\low\frac{|\nabla \Psi|^{2}}{\Psi} -4\mu_{2}\low ~e\cdot\nabla\Psi +\frac{\Psi f(x,\low)}{\low\left(1-\low\right)}-\eta(x)\Psi > 0,
\end{equation}
then nonlinear selection is realized.
 \end{corollary}
\begin{proof} Computations, similar to the ones performed on $\up$ (Proof of Corollary \ref{co2.2}), yield that
\[\begin{array}{ll}
L_{c}\low+f(x,\low)&
=\ds{\frac{\low\left(1-\low\right)}{\Psi}}\left\{\ds{\mu_{2}^{2}(1-2\low)\Psi-2\low\frac{|\nabla \Psi|^{2}}{\Psi}+\Delta\Psi+2\mu_{2}(1-2\low)e\cdot\nabla\Psi} \right.\\
&\left. +q\cdot\nabla\Psi+\mu_{2}q\cdot e\Psi-c\mu_{2}\Psi+\frac{\Psi f(x,\low)}{\low\left(1-\low\right)}\right\}\vspace{10pt}\\
&=\ds{\frac{\low\left(1-\low\right)}{\Psi}}\left\{\ds{k(\mu_{2})\Psi-c\mu_{2}\Psi-2\mu_{2}^{2}\low \Psi -2\low\frac{|\nabla \Psi|^{2}}{\Psi}}\right.\\
&\left. -4\mu_{2}\low ~e\cdot\nabla\Psi +\frac{\Psi f(x,\low)}{\low\left(1-\low\right)}-\eta(x)\Psi\right\}\vspace{10 pt}\\
&=\ds{\frac{\low\left(1-\low\right)}{\Psi}}\left\{-2\mu_{2}^{2}\low \Psi -2\low\frac{|\nabla \Psi|^{2}}{\Psi} -4\mu_{2}\low ~e\cdot\nabla\Psi +\frac{\Psi f(x,\low)}{\low\left(1-\low\right)}-\eta(x)\Psi\right\},
\end{array}\]
since $k(\mu_{2})-c\mu_{2}=0.$ Hence, the result  follows from Theorem \ref{them8} by taking $V=(1-\eta) \underline \phi,$ with a sufficiently small $\eta$.
\end{proof}

Theorem \ref{them5} gives a lower estimate for the minimal speed. We can also provide an upper estimate for the minimal speed, when the nonlinear selection is realized.

 \begin{thm}[Upper bound for the minimal speed]\label{them5a}
For $c_2>c_0$, suppose that there exists a function $V_2(s,x)$ satisfying
   \begin{equation}
  0 < V_2(s,x) < 1, \quad   \limsup_{s\rightarrow \infty} V_2(s,x) \le 1, \quad V_2(s,x)=\Psi_{{\mu_2(c_2)}}(x)e^{\mu_2(c_2)s}   ~ \text{ as }~ s\rightarrow -\infty,     \label{non1a}
   \end{equation}
   and
   \begin{equation} L_{c_1}V+f(x,V)\le 0, \label{c1a1}
\end{equation}
  where $\mu_2(c_2)$ is defined in (\ref{mu2a}). Then, $c^* \leq c_2$.
  \label{them9}
   \end{thm}

   \begin{proof}The proof follows from the comparison principal and it is similar to that of Theorem \ref{them2}, as long as we choose the initial function $\phi(x)$, in (\ref{2b}), less than $V_2(x\cdot e,x)=\left.V(x\cdot e+ct,x)\right|_{t=0}$.
   \end{proof}

   \section{Application}\label{apps}
   In this section, we consider a simple case, where $N=1$ and the advection $q$ is a constant. We will show how our results reflect on the  equation (\ref{1}), which now reads
   \begin{equation}
u_t =u_{xx}+q  u_x+f(x,u), \quad \label{3.1a}
t\in\mathbb{R},~~x \in\R.
\end{equation}
We consider a nonlinearity $f$  in modified KPP-Fisher class, with the Allee effect. Namely, 
\begin{equation}\label{f}
  f(x,u)=u(1-u)(1+a(x)u),
\end{equation}
where $a(x)$ is a positive periodic function. Since $\eta(x)=\partial_uf(x,0)=1$, the principal eigenfunction of \eqref{Leq} is $\Psi=1$ and the principal eigenvalue is $k(\lambda)=q\lambda+\lambda^2+1$ for all $\lambda>0.$ Thus, the linear speed is
\begin{equation}\label{3.5a}
  c_0:=\min_{\lambda>0}\frac{k(\lambda)}{\lambda}=\min_{\lambda>0}(q+\lambda+\frac{1}{\lambda})=q+2.
\end{equation}
Moreover, for any $c>q+2$, from (\ref{mu2a}), we have the equation
\begin{equation}\label{4.5}
  \lambda^2+(q-c)\lambda+1=0.
\end{equation}
The two roots are
\begin{equation}\label{4.6}
  \lambda=\mu_1(c)=\frac{c-q-\sqrt{(c-q)^2-4}}{2} ~\text{ and }~  \lambda=\mu_2(c)=\frac{c-q+\sqrt{(c-q)^2-4}}{2}.
\end{equation}
When $c=c_0$, we have $\mu_1(c_0)=\mu_2(c_0)=1$.\\

Applying Corollaries  \ref{co2.2} and \ref{them6}, we obtain that $$c_{\min}=q+2,~\text{ if }~a(x)\le 2~\text{ for all }x,$$ and $$c_{\min}>q+2,~\text{ if }~a(x)> 2~\text{ for all }x.$$ 
Lastly, in the case where $a(x)>2\text{ for all }x $,  let
\begin{equation}
  m=\min a(x) ~\text{ and }~M=\max a(x).
\end{equation}
Then, it can be derived that the minimal speed satisfies
\begin{equation}\label{4.8}
 q+\sqrt{\frac{m}{2}}+\sqrt{\frac{2}{m}} < c_{\min}<q+\sqrt{\frac{M}{2}}+\sqrt{\frac{2}{M}}.
\end{equation}
This provides  upper and lower estimates for the minimal speed.
\section{Summary}
In this paper, we studied the
speed selection for reaction diffusion equations in heterogeneous environments. The key feature of the nonlinear selection of the minimal speed was unveiled. We  proved that the well-known  minimal speed $c^*$ is \emph{linearly} selected if we can find an upper solution with the linear speed. We also proved that  $c^*$ is \emph{nonlinearly} selected if we can find a lower solution with a  faster decay rate, at some speed that is greater than the linear speed $c_0$. As applications to these results, upper/lower bounds of the minimal speed were provided in the case of \emph{nonlinear} selection.

\section{acknowledgement}
Mohammad El Smaily was partially supported by the  Canadian Natural Sciences and Engineering Research Council through the NSERC Discovery Grant  (RGPIN-2017-04313).
Chunhua Ou was partially supported by the  NSERC Discovery Grant (RGPIN04509-2016).

\end{document}